\documentclass{amsart}
\usepackage{amssymb}
\newtheorem{theorem*}{Theorem}
\newtheorem{lemma}{Lemma}
\newtheorem{definition}{Definition}
\begin{document}
\author[A.V. Petukhov]{Alexey V. Petukhov}
\title[Affinity criterion for the quotient of an algebraic group by a one-dimensional subgroup]{Affinity criterion for the quotient of an algebraic group by a
one-dimensional subgroup}
\begin{abstract}
  In this work we show that the homogeneous space of an affine algebraic group $G$ by a one-dimensional unipotent subgroup $H$ is affine if and only if the subgroup is not contained in any reductive subgroup
of $G$.
\end{abstract}
\maketitle
    Let $G$ be an affine algebraic group over an algebraically closed field of characteristic zero and $H$ a closed subgroup of $G$. It is well known that the homogeneous space $G/H$ has canonical structure of a quasiprojective $G$-variety, but it is not known for which pairs $(G,H)$ the variety $G/H$ is affine.
The answer in the case of a reductive group $G$ was found by Matsushima ~\cite{Mats}. The quotient space is affine variety if and only if $H$ is reductive. In ~\cite{MM} it is proved that $G/H$ is affine if and only if $G/H^u$ is affine, where $H^u~$is the unipotent radical of $H$. Moreover, it is easy to check that if $G/H^u$ is affine, then $G/F$ is affine for any unipotent subgroups $F \subset H^u$.
The main goal of this work is to determine when the homogeneous space by a one-dimensional unipotent subgroup $H$ is affine.
\begin{theorem*} The homogeneous space $G/H$ of an affine algebraic group $G$ by a one-dimensional unipotent group $H$ is affine if and only if $H$ is not contained in any reductive subgroup of $G$. \end{theorem*}
We begin with some auxiliary lemmas.

Let $A\supset B\supset C$ be algebraic groups.
\begin{lemma} Let $A/B$ and $B/C$ be affine varieties. Then $A/C$ is an affine variety. ~\cite[стр 45]{Gr} \end{lemma}
\begin{lemma} Let $A/B$ be an affine variety. Then $A/C$ is affine if and only if $B/C$ is affine. \end{lemma}
\begin{proof} Let $B/C$ be an affine variety. It follows from Lemma 1 that $A/C$ is an affine variety. On the other hand, if $A/C$ is affine, then $B/C$ is the image of the closed $C$-invariant subvariety $B$ by the geometric quotient $~A \rightarrow A/C$.\end{proof}
\begin{definition}\upshape Let $H$ be an algebraic subgroup of an algebraic group $G$. A closed subvariety of G is said to be a {\slshape section} of the action $H:G$ if it intersects all orbits of $H$ in one point and intersects all irreducible components of G by irreducible subvarieties.\end{definition}
{\bfseries Remark.} If there exists a smooth section and a quasiaffine quotient, then the quotient is affine since a bijective morphism of smooth varieties is an isomorphism.
 
{\bfseries Basic Lemma.} Let $G = L \rightthreetimes V$, where L is a reductive group with a representation $\rho:L\rightarrow $GL$(V)$. Let $h$ be a unipotent element of $G$. Let $H$ be a minimal algebraic subgroup that contains it. Then
either there exists an element ${h'} \in L$ conjugated to $h$, or there exists a section of the action~$H:G$.
\begin{proof}Decompose $h$ as $l\cdot v,l\in L, v\in V$. We should examine two cases:

1) If $v\in~$Im$(\rho (l)-{\text Id})$, then there is $f \in V$ such that ${h'}={f^{-1}}{h}{f}\in~L$ 

2) If $v \notin$ Im $(\rho (l)-{\text Id})$ then there exists a hypersurface $V' \subset V$ containing Im$(\rho (l)-{\text Id})$ and such that $V=\langle v\rangle  \bigoplus V'$.
Let $K=LV^{'}$ be the preimage of $V'$ under the morphism $G\rightarrow L \backslash G$. Then $K$ is a section of the action $H:G$. Since $\rho (l)$ preserves Im($\rho (l)-{\text Id}$) and its action on the quotient $V/$Im($\rho(l)-{\text Id}$) is trivial, we have the action $H:V/$Im($\rho(l)-{\text Id}$), where $H$ acts by parallel translations. 
It is easy to see that for this action there is a section of $G$ that is a hypersurface; its preimage we denote by $V'$. Since the action $H$: $V/V'$ is free and transitive, the preimage of zero under the morphism $G \rightarrow L\backslash G \rightarrow (L\backslash G)/V'$ is a section. \end{proof}

\begin{proof}[{\bfseries Proof of Theorem 1.}]	If a reductive subgroup $S$ of $G$ contains $H$, then $S/H$ is not affine (Matsushima's criterion). Hence $G/H$ is not affine (Lemma 2). Now we prove that if such subgroup does not exist then the quotient is not affine.

	In our proof we use induction on dimension of the unipotent radical $G^u$ of $G$. If dim$~G^u$$=0$, then it is holds true. Assume that dim$~G^u > 0$.
Let $L$ be a maximal reductive subgroup of $G$ and $h$ be a nontrivial element of $H$. It has a presentation $h = l \cdot r$, where $l\in L$ and $r\in G^u$. We will see that $L$ conjugates the vector space $V=G^u/(G^{u},G^{u})$ 
and by the Basic lemma either $h$ is
conjugated to some element $h'$, where $h'$ is contained in $L \rightthreetimes (G^{u},G^{u})$, or there is a section of the action $H$ on $L \rightthreetimes (G^u/(G^{u},G^{u})) $ and consequently on $G$. Let $H'$ be the minimal algebraic subgroup that contains element $h'$. Then $G/H \cong G/ H'$
but $(L \rightthreetimes (G^u /(G^{u},G^{u})))/L$ is affine and $(L\rightthreetimes (G^u,G^u))/H'$ is affine by the inductive assumption. Using Lemma 2 we get that $G/H$ is an affine variety. \end{proof}
The main result of this work shows that the quotient $G/H$, where $H$ is a one-dimensional subgroup, is affine if and only if any group conjugated to $H$ does not intersect a maximal reductive subgroup of $G$. This fact arises a question: "Is this condition sufficient for  $G/H$ to be an affine variety?".
This question was solved negatively in ~\cite{Wink}.

The author thanks I.V. Arzhantsev and E.B.Vinberg for their effective questions and remarks. Without their help the work would not be published. 

\end{document}